\documentclass[11pt]{article}
\usepackage[letterpaper,left=1in,right=1in,top=1in,bottom=1in]{geometry}
\usepackage[mathscr]{eucal}
\usepackage{epsfig,epsf,psfrag}
\usepackage{amssymb,amsfonts,latexsym}
\usepackage{amsmath}
\usepackage{graphicx}
\usepackage{epstopdf}
\usepackage{bm,xcolor,url}
\usepackage{fixltx2e}
\usepackage{array}
\usepackage{verbatim}
\usepackage[noend]{algpseudocode}
\usepackage{cite}
\usepackage{algorithm}
\usepackage{verbatim}
\usepackage{textcomp}
\usepackage{mathrsfs}
\usepackage[referable]{threeparttablex}
\usepackage{subfig}
\usepackage{amsthm}
\usepackage{enumerate}
\usepackage{footnote}
\usepackage{enumitem}
\usepackage{xr}
\usepackage{mathtools}
\catcode`~=11 \def\UrlSpecials{\do\~{\kern -.15em\lower .7ex\hbox{~}\kern .04em}} \catcode`~=13 

\allowdisplaybreaks[3]

\newcommand{\tnorm}[1]{{\left\vert\kern-0.25ex\left\vert\kern-0.25ex\left\vert #1 
    \right\vert\kern-0.25ex\right\vert\kern-0.25ex\right\vert}}
\newcommand{\tnormt}[1]{{\vert\kern-0.25ex\vert\kern-0.25ex\vert #1 
    \vert\kern-0.25ex\vert\kern-0.25ex\vert}}

\newcommand{\R}{\bbR}

\newcommand{\norm}[1]{\left\Vert#1\right\Vert}
\newcommand{\normt}[1]{\Vert#1\Vert}

\newcommand{\nn}{\nonumber}

\newcommand{\nt}{\addtocounter{equation}{1}\tag{\theequation}} 

\newcommand{\dom}{\mathsf{dom}\,}

\newcommand{\ri}{\mathsf{ri}\,}
\newcommand{\where}{\mathrm{where}}

\newcommand{\andd}{\mathrm{and}}

\newcommand{\conv}{\mathsf{conv}\,}

\newcommand{\diam}{\mathsf{diam}}

\newcommand{\calY}{\mathcal{Y}}

\newcommand{\rmP}{\mathrm{P}}

\newcommand{\rms}{\mathrm{s}}

\newcommand{\bbR}{\mathbb{R}}

\DeclareMathAlphabet{\mathbsf}{OT1}{cmss}{bx}{n}

\newcommand{\haty}{\hat{y}}

\newcommand{\bart}{\bar{t}}

\newcommand{\barx}{\bar{x}}
\newcommand{\bary}{\bar{y}}

\newcommand{\ceil}[1]{\lceil{#1}\rceil}
\newcommand{\floor}[1]{\lfloor{#1}\rfloor}
\newcommand{\ip}[2]{\left\langle{#1},{#2}\right\rangle}
\newcommand{\ipt}[2]{\langle{#1},{#2}\rangle}

\newcommand{\eqa}{\stackrel{\rm(a)}{=}}

\DeclareMathOperator*{\argmin}{arg\,min}

\DeclareMathOperator{\st}{s.t.}

\newtheorem{theorem}{Theorem} 
\newtheorem*{theorem*}{Theorem}
\newtheorem{lemma}{Lemma}

\newtheorem{prop}{Proposition}
\newtheorem{corollary}{Corollary}
 
\newtheorem*{assump*}{Assumption}

\theoremstyle{definition}

\theoremstyle{remark}
\newtheorem{remark}{Remark}

\newcommand{\qednew}{\nobreak \ifvmode \relax \else
      \ifdim\lastskip<1.5em \hskip-\lastskip
      \hskip1.5em plus0em minus0.5em \fi \nobreak
      \vrule height0.75em width0.5em depth0.25em\fi}

\title{On the Optimal Linear Minimization Oracle Complexity of Active-Hull Methods for Minimizing Convex H\"older Smooth Functions}
\usepackage[colorlinks=true,breaklinks=true,bookmarks=true,urlcolor=blue,
     citecolor=blue,linkcolor=blue,bookmarksopen=false,draft=false]{hyperref}
\author{Renbo Zhao
\thanks{Tippie College of Business, University of Iowa, Iowa City, IA 52242 ({mailto:  renbo-zhao@uiowa.edu}).}
}
\usepackage[normalem]{ulem}

\usepackage{algorithm}

\numberwithin{equation}{section}
\numberwithin{definition}{section}
\numberwithin{prop}{section}
\numberwithin{lemma}{section}
\numberwithin{theorem}{section}
\numberwithin{remark}{section}
\numberwithin{corollary}{section}

\newcommand{\Df}{ D_f}
\newcommand{\Dg}{ D_g}

\begin{document}

\maketitle

\begin{abstract}
We consider the convex problem  $(P): {\min}_{y\in \bbR^n}\, f(Ay)+g(y)$, where $A:\bbR^n\to \bbR^m$ is a linear operator, $f$ is convex and $(M,\nu)$-H\"older smooth on $\bbR^m$, and $g$ is proper, closed convex on $\bbR^n$,  and admits a (generalized) linear minimization oracle (LMO). An active-hull  (AH) method calls the LMO in each iteration and outputs a point that lies in the convex hull of all the past LMO returns. We show that under the Euclidean geometry (and in the high-dimensional regime), the optimal LMO complexity for the class of AH methods to solve $(P)$ is $O\big(M^{{2}/({1+\nu})}D^2\varepsilon^{-{2}/({1+\nu})}\big)$, where $D$  denotes the diameter of the feasible region and $\varepsilon$ denotes the objective sub-optimality gap. In particular, the optimal LMO complexity is achieved by a single-loop first-order primal-dual splitting method. Since this method requires knowledge of the problem parameters  (e.g., $M$, $D$ and $\nu$), we propose two meta parameter-search algorithms that aim to resolve this issue.  These algorithms are parameter-free, but come with the price of an additional log-factor in their LMO complexities.  
\end{abstract}

\section{Introduction} \label{sec:intro}

Consider the following convex composite optimization problem: 
\begin{equation}
p^*:={\min}_{y\in \bbR^n}\; \{p(y):=f(Ay)+g(y)\}, \tag{P} \label{eq:P}
\end{equation}
where $A:\bbR^n\to \bbR^m$ is a linear operator and $f:\bbR^m\to \bbR$ is convex and $(M,\nu)$-H\"older smooth on $\bbR^m$ for some $M>0$ and $\nu\in[0,1]$, i.e., 
\begin{equation}
0\le f(x)- f(z) - \ip{s_z}{x-z} \le \frac{M}{1+\nu}\norm{x-z}^{1+\nu}, \quad \forall\,s_z\in \partial f(z),\quad \forall\, z,x\in \bbR^m. \label{eq:holder}
\end{equation}
Throughout this paper, $\normt{\cdot}$ 
denotes  the {\em Euclidean} norm. In particular, we call $f$ {\em nonsmooth} (and Lipschitz) if $\nu=0$ and {\em smooth} if $\nu=1$. 
Also, $g:\bbR^n\to \bbR\cup\{+\infty\}$ is a proper, closed and convex (p.c.c.) function.  Additionally, we assume that 
\begin{enumerate}[label=(A\arabic*)]
\item \label{assump:LMO} For any $\lambda\in \bbR^n$, the following optimization problem 
\begin{equation}
 {\min}_{x\in \bbR^n}\,\{\ip{\lambda}{y}+g(y)\}, 
  \label{eq:LMO}
\end{equation}
admits an easily computable optimal solution. We shall call the oracle that outputs an solution of~\eqref{eq:LMO} the   {\em (generalized) linear minimization oracle} (LMO) of $g$.
\item The set $A(\dom g)$ is bounded, where $\dom g:= \{y\in\bbR^n:g(y)<+\infty\}$ denotes the domain of $g$. Consequently, define
\begin{equation}
D:= {\sup}_{x,\,z\in A(\dom g)}\;\normt{x-z}<+\infty.  \label{eq:def_D}
\end{equation}
\end{enumerate}
In this work, we 
consider the class of LMO-based active-hull (LMO-AH) methods for solving~\eqref{eq:P}, which is shown in Algorithm~\ref{algo:LMO-AH}. This class of methods was proposed in the seminal work of Lan~\cite{Lan_13}  for solving the following subclass of the problems in~\eqref{eq:P}:
\begin{equation}
{\min}_{y\in \calY}\; f(y), \tag{$\rm P_0$} 
\label{eq:P0} 
\end{equation}
where $\calY\ne\emptyset$ is a convex compact set and $f$ is either nonsmooth or smooth (which corresponds to either $\nu=0$ or $\nu=1$).  
As remarked in~\cite{Lan_13}, this class of methods is fairly broad. 
Among them, an important subclass is the 
 {\em first-order} LMO-AH methods, wherein the dual variable $\lambda_i$ is constructed by using only the first-order oracle (FO) of $f$ and LMO of $g$. In fact, 
 these methods  include the Frank-Wolfe (FW) method~\cite{Frank_56} and many of its variants (see~\cite{Lan_13,Jaggi_13,Lacoste_15,Lan_16,Zhao_22} and references therein) as special cases.  That said, the entire class of LMO-AH methods need not be first-order, and $\lambda_i$ can be constructed via various other oracles of $f$, such as its LMO, proximal oracle or higher-order oracle. 
Indeed, there are only two defining features of the LMO-AH methods: i) at each iteration, the iterate $y_{i+1}$ is returned by the LMO, and ii) the final output $\bary_k$ lies in the convex hull of all the past iterates $\{y_0,\ldots,y_k\}$, which gives rise to the qualifier ``active-hull''. 

For a feasible point $y\in \dom g$, we call it an $\varepsilon$-optimal solution of~\eqref{eq:P} if  $p(y) - p^*<\varepsilon$. 
The main focus of this work is to investigate the optimal LMO complexity of the class of LMO-AH methods for finding an $\varepsilon$-optimal solution of~\eqref{eq:P}. 
Let us start with lower bounds of this optimal complexity. 
For the cases where $\nu=0$ or $\nu=1$,  Lan~\cite{Lan_13} established the following lower LMO-complexity bounds for LMO-AH methods to find  an $\varepsilon$-optimal solution of~\eqref{eq:P0}: 
\begin{align}
\Omega\left(\min\left\{n,M^{{2}}D^2\varepsilon^{-{2}}\right\}\right)\quad \mbox{for }\nu=0, \quad \andd \quad  
\Omega\left(\min\left\{n, MD^2\varepsilon^{-1}\right\}\right)\quad \mbox{for }\nu=1.  \label{eq:lb_Lan}
\end{align}
Since the problems in~\eqref{eq:P0} form a subclass of the problems in~\eqref{eq:P}, the lower bounds in~\eqref{eq:lb_Lan} also apply to~\eqref{eq:P}. However, to our knowledge, when $\nu\in(0,1)$, no lower bound has been established. As a first contribution in this paper, we show that across all values of $\nu\in[0,1]$, the LMO complexity for LMO-AH methods to find  an $\varepsilon$-optimal solution of~\eqref{eq:P} should be at least
\begin{align}
\Omega\left(\min\left\{n,M^{\frac{2}{1+\nu}}D^2\varepsilon^{-\frac{2}{1+\nu}}\right\}\right). \label{eq:lb_holder}
 \end{align}
Indeed, this lower bound can be regarded as an ``interpolation'' between the lower bounds at $\nu=0$ and $\nu=1$, which are shown in~\eqref{eq:lb_Lan}.

Of course, at this point, one may naturally wonder if any of the lower complexity bounds above are optimal, i.e., if any of them can be achieved by some LMO-AH method. In fact, when $\nu=1$, it is well-known that the (generalized) Frank-Wolfe methods (see e.g.,~\cite{Lan_13,Jaggi_13,Bach_15,Freund_16,Nest_18,Ghad_19,Pena_23}) can achieve the lower complexity bound in~\eqref{eq:lb_Lan}. The situation becomes very different in the case of $\nu=0$, since it is well-known from~\cite[Example 1]{Nest_18} that the 
FW method may not convergence to the optimal solution of~\eqref{eq:P0} with the subgradient of $f$. Due to this, many works~\cite{Cox_14,He_15a,Pier_14,Lan_16} additionally assume that $f$ has a ``proximal-friendly'' dual representation. Since the domain of this dual representation is bounded (as $f$ is Lipschitz on $\bbR^m$), the problem in~\eqref{eq:P} can be effectively cast as a convex-concave saddle-point problems (SPPs) with bilinear coupling and bounded primal and dual domains. As such, these works propose to modify the existing projection-based methods for solving SPPs (such as Nesterov's smoothing~\cite{Nest_05}, dual Mirror descent~\cite{Nemi_79} and Mirror-Prox~\cite{Nest_04}) to accommodate the LMO 
in~\eqref{eq:LMO}. It turns out that all of the methods in~\cite{Cox_14,He_15a,Pier_14,Lan_16}  are LMO-AH methods, and they all achieve the lower complexity bound in~\eqref{eq:lb_Lan}. With above said, recently there were two notable first-order LMO-AH methods~\cite{Theku_20,Asgari_22} that achieve the the lower complexity bound in~\eqref{eq:lb_Lan} for solving~\eqref{eq:P0}  without assuming a ``proximal-friendly'' dual representation of $f$. The first method~\cite{Theku_20} 
leverage Moreau smoothing~\cite{Moreau_65} and the inexact accelerated proximal gradient method~\cite{Dev_14}, wherein the FW method was used 
to solve the (Euclidean) projection sub-problem onto $\calY$. In contrast to the primal and multi-loop nature of the first method, the second method in~\cite{Asgari_22} is primal-dual, single-loop and quite simple (where each iteration only consists of three updates). Indeed, the authors 
 explicitly introduce a primal slack variable $x$ and solve the following equivalent problem instead:
\begin{align}
{\min}_{x,y}\; f(x)\qquad \st\quad  x=y,\quad y\in\calY. 
\end{align}
It is worth mentioning that 
the simplicity and single-loop nature of this method are particularly appealing, both in theory and practice. 

 
 Up to this point, it is clear that both lower LMO complexity bounds in~\eqref{eq:lb_Lan}, which correspond to the cases $\nu=0$ and $\nu=1$, respectively, turn out to be optimal. However, in the case of $\nu\in(0,1)$, the situation becomes much unclearer. In fact, in this case, the literature on the LMO-AH methods 
 is much scarcer. 
 To start with, it is widely known that (see e.g.,~\cite{Dem_70,Dunn_78,Nest_18,Ghad_19,Pena_23}) the LMO complexity of the (generalized) FW method for solving~\eqref{eq:P} is
\begin{equation}
O(M^{1/\nu}D^{1+1/\nu}\varepsilon^{-1/\nu}), \qquad \mbox{for all}\;\; \nu\in(0,1], \label{eq:comp_FW_nu}
\end{equation}
which clearly does not match the lower bound in~\eqref{eq:lb_holder}. 
This complexity has been improved in the recent seminal work by Ouyang and Squires~\cite{Ouyang_23}, where the authors proposed a  universal conditional gradient sliding (UCGS) method, and showed that the LMO complexity of this method is
\begin{equation}
O\Big( M^{\frac{4}{1+3\nu}}D^{\frac{4(1+\nu)}{1+3\nu}}\varepsilon^{-\frac{4}{1+3\nu}}\Big), \qquad \mbox{for all}\;\; \nu\in[0,1]. \label{eq:comp_UCGS_nu}
\end{equation}
As observed by the authors, this complexity does not match the optimal one when $\nu=0$, and in their concluding remarks, they pose the following question:
\begin{center}
Is there a method that can improve the LMO complexity developed for the case where \(\nu \in (0,1]\),
 and can a universal method covering all cases \(\nu \in [0,1]\) be developed? 
\end{center}
In the light of our discussions above, we can ask a more ambitious question: 
\begin{center}\em 
Is there a  method that can achieve the lower LMO complexity bound in~\eqref{eq:lb_holder} for  all \(\nu \in [0,1]\)?
\end{center}

In this paper, we shall provide an affirmative answer to this question. We propose a single-loop first-order primal-dual splitting method (in Algorithm~\ref{algo:CG}), which is an instance of Algorithm~\ref{algo:LMO-AH}, that has an LMO complexity matching the lower bound in~\eqref{eq:lb_holder} for  all \(\nu \in [0,1]\) (and for the dimension $n$ sufficiently large). 
As a result, we simultaneously establish the optimality of the lower LMO-complexity bound in~\eqref{eq:lb_holder} and identify an optimal LMO-AH method that achieves this complexity. Since this method requires knowledge of the problem parameters (including $M$, $D$ and $\nu$), we propose two meta parameter-search algorithms that aim to resolve this issue.  These algorithms are parameter-free, but come with the price of an additional log-factor in their LMO complexities. 
%
%


\begin{algorithm}[t]
\caption{LMO-Based Active-Hull Methods}
\label{algo:LMO-AH}
\begin{algorithmic}
    \State {\bf Input}: $y_0 \in \dom g$, total number of iterations $K\ge 1$ and problem parameters (e.g., $M$, $\nu$, $D$)
    \For{\(k=0,1,2,\ldots,K-1\)}
        \State Construct dual variable $\lambda_k\in \bbR^n$ 
        \State Compute $y_{k+1} \in \argmin_{y\in\bbR^n} \; \ip{\lambda_k}{y} + g(y)$
   \EndFor
        \State {\bf Output}: $\bary_K \in \conv\{y_0,\ldots,y_K\}$
\end{algorithmic}
\end{algorithm}

\section{A Lower bound}

Let us present a lower LMO-complexity bound of any LMO-AH method. Notation-wise, let $e_i$ denote the $i$-th standard coordinate vector for $i\in[n]$ and $e:= \sum_{i=1}^n e_i$. 

\begin{theorem}
\label{thm:simplex-lower-bound}
Let $0\le \nu\le1$, $M>0$, $D>0$ and $n\ge 4$. 
 There exists an instance of~\eqref{eq:P0} 
such that for any LMO-AH method 
with at most $K$ 
LMO calls, its output $\bary_K$ satisfies that 
\begin{equation}
 f(\bary_K) -p^*  \ge C_\nu{MD^{1+\nu}}{(K+1)^{-\frac{1+\nu}{2}}},
 \qquad \forall\, 1\le K\le \floor{n/2}-1, 
 \label{eq:simplex-gap-lower-bound}
\end{equation}
where $C_\nu>0$ is some constant that only depends on $\nu$.  In other words, for any 
LMO-AH method to output an $\varepsilon$-optimal solution of~\eqref{eq:P0}, the number of 
LMO calls 
has to be at least
\begin{align}
\left\lceil\min\left\{\floor{{n}/{2}},\;C_\nu^{\frac{2}{1+\nu}} M^{\frac{2}{1+\nu}}D^2\varepsilon^{-\frac{2}{1+\nu}} \right\}\right\rceil -1 . 
\label{eq:complexity}
\end{align}
\end{theorem}

\begin{proof}
In~\eqref{eq:P0}, let  $\calY:= r\Delta_n$, where $\Delta_n:= \{y\ge 0: e^\top y =1\}$ and $r>0$, and 
\begin{equation}
 f(y):=\frac{M}{2^{1-\nu}(1+\nu)}\norm{y}^{1+\nu}, \quad \forall\, y\in\R^n. 
 \label{eq:def_f}
\end{equation}
From~\cite[Theorem~6.3]{Rodom_20}, we know that $f$ is $(M,\nu)$-H\"older smooth on $\bbR^n$. Also, we know that 
\begin{align}
&
\diam (\calY) = ({\sup}_{y,y'\in r\Delta_n}\;\normt{y-y'}^2)^{1/2}\le \sqrt{2} r=D \quad\mbox{and}\quad \diam (\calY) \ge \normt{re_i - re_j} = \sqrt{2} r = D,\nn
\end{align}
and hence $\diam (\calY) =  \sqrt{2} r=D$. 
Furthermore, it is easy to verify that the (unique) optimal solution of this problem is given by 
\begin{equation}
 y^*=\frac{r}{n}e\quad \Longrightarrow\quad p^*=f(y^*)
 =\frac{M}{2^{1-\nu}(1+\nu)}
   r^{1+\nu}n^{-\frac{1+\nu}{2}}.
 \label{eq:opt_value}
\end{equation}
%

Now, let $y_0 = r e_1$. In addition, let the output of the LMO be $\{e_1,\ldots,e_n\}$, so $y_k = e_i$ for some $i\in[n]$, for all $k\ge 0$. Fix some $1\le k\le \floor{n/2}-1$. After $k$ LMO calls, since $\bary_k\in \conv\{y_0,\ldots,y_k\}$,   we know that $\normt{\bary_k}_0\le k+1$. Let $\bary_k^+$ be the sub-vector of $\bary_k$ with positive entries, and we have 
\begin{equation}
\normt{\bary_k} = \normt{\bary_k^+}\ge \frac{\normt{\bary_k^+}_1}{\sqrt{k+1}} = \frac{\normt{\bary_k}_1}{\sqrt{k+1}} = \frac{r}{\sqrt{k+1}}. \label{eq:lb_normy}
\end{equation}
Therefore, by~\eqref{eq:def_f},~\eqref{eq:opt_value} and~\eqref{eq:lb_normy}, we have 
\begin{align}
 f(\bary_k)-p^*  &\ge\frac{M}{2^{1-\nu}(1+\nu)}r^{1+\nu}  \left(  (k+1)^{-\frac{1+\nu}{2}}-n^{-\frac{1+\nu}{2}}
 \right).
\end{align}
Since $k\le \floor{n/2}-1$, we have  $n\ge 2(k+1)$, and 
\begin{align}
 f(\bary_k)-p^*  &\ge\frac{M}{2^{1-\nu}(1+\nu)}r^{1+\nu}  (1 - 2^{-\frac{1+\nu}{2}}) (k+1)^{-\frac{1+\nu}{2}}\\
 &= C_\nu  MD^{1+\nu} (k+1)^{-\frac{1+\nu}{2}}, \qquad \where \quad 
 C_\nu:=\frac{2^{\frac{1+\nu}{2}} - 1}{4 (1+\nu)}. 
\end{align}
As a result, if $k\le \min\{\floor{n/2},C_\nu^{\frac{2}{1+\nu}}M^{\frac{2}{1+\nu}}D^2\varepsilon^{-\frac{2}{1+\nu}}\}-1$, then 
\begin{align}
 f(\bary_k)-p^*  \ge  C_\nu  MD^{1+\nu}\max\{(n/2)^{-\frac{1+\nu}{2}},C_\nu^{-1}M^{-1}D^{-(1+\nu)}  \varepsilon \}\ge \varepsilon.
\end{align}
Therefore, if $f(\bary_k)-p^*<\varepsilon$, then we must have $k> \min\{\floor{n/2},C_\nu^{\frac{2}{1+\nu}}M^{\frac{2}{1+\nu}}D^2\varepsilon^{-\frac{2}{1+\nu}}\}-1$. 
\end{proof}

\section{A First-Order Primal-Dual Method and Its Rate Analysis} 

\begin{algorithm}[t]
\caption{A First-Order Primal-Dual Splitting Method}\label{algo:CG}
\begin{algorithmic}[1]
	\State {\bf Input:}  $y_{0}\in \dom g$, total number of iterations $K\ge 1$, parameters $\{\eta_k\}_{k\ge 0}$ and $\{\rho_k\}_{k\ge 0}$  
	\State  {\bf Initialize}: $x_0 = Ay_0$ and $\lambda_0\in \partial f(x_0)$ \label{line:controlled-init}
	\For{$k=0,1,2,\dots,K-1$}\vspace{-1ex}
		\begin{align}
 & y_{k+1} \in\arg\min_{y\in \R^n}g(y)+\ip{\lambda_k}{Ay}
 \label{eq:y-update}\\
 &\mbox{Compute } s_k\in \partial f(x_k)\\
 & x_{k+1}  =\arg\min_{x\in \R^m}   \ip{s_k-\lambda_k}{x}    +\frac{\eta_k}{2}\norm{x-x_k}^2    +\frac{\rho_k}{2}\norm{Ay_{k+1}-x}^2 \quad 
 \label{eq:x-update}\\
 &\lambda_{k+1}  =\lambda_k+\rho_k(Ay_{k+1}-x_{k+1})\\[-4ex]\nn
 \label{eq:dual-update}
\end{align}
 \State {\bf Output}: $\bary_K \in \conv\{y_1,\ldots,y_K\}$
	\EndFor
\end{algorithmic}
\end{algorithm}

\subsection{The Method} \label{sec:method}

To introduce this method, let us write down the Fenchel dual of~\eqref{eq:P}: 
\begin{equation}\label{eq:D}
-d^*:= -{\min}_{\lambda\in \bbR^n}\, \{d(u):=g^*(-A^*\lambda)+ f^*(\lambda)\},  \tag{D} 
\end{equation}
where $A^*:\bbR^m\to\R^n$ denotes the adjoint of $A$. 
Note that Assumption~\ref{assump:LMO} 
amounts to assuming that $\partial g^*(-\lambda)\ne \emptyset$ for all $\lambda\in\bbR^n$, which in turn amounts to assuming that $\dom g^* = \bbR^n$. 
This, together with the H\"older smoothness of $f$ on $\bbR^m$, indicates that both $\ri \dom f \cap\ri\dom \Psi\ne \emptyset$ and $(-\ri \dom f^*) \cap\ri\dom \Psi^*\ne \emptyset$. Hence by the Fenchel duality theorem (cf.~\cite[Theorem~31.1]{Rock_70}), we know that $ p^* = -d^* \in\bbR$, and both~\eqref{eq:P} and~\eqref{eq:D} have at least one optimal solution, which we denote by $y^*\in \dom g$ and $\lambda^*\in \dom f^*$, respectively. Consequently,  
we know that (cf.~\cite[Theorem~3.50]{Peyp_15})
\begin{equation}
 \lambda^*\in \partial f(A y^*), \quad -A^*\lambda^* \in \partial g(y^*).  \label{eq:opt_Fenchel_dual}
\end{equation}

Now, let us write~\eqref{eq:P} equivalently as 
\begin{align}
{\min}_{x,y}\; f(x) + g(y)\qquad \st\quad  Ay=x. \tag{$\rmP_\rms$} \label{eq:Ps} 
\end{align}
From our discussions above, it is clear that~\eqref{eq:Ps} has an optimal solution $(x^*,y^*)$, where $x^* = Ay^*$. 
Instead of solving~\eqref{eq:P} directly,  we aim to solve~\eqref{eq:Ps}, but of course, our eventual goal is still providing convergence rate guarantees to~\eqref{eq:P}.  
The motivation behind this is simply to {\em separate} the 
FO of $f$ and the LMO of $g$. In the case of $\nu=0$, this separation seems important for a LMO-AH method to obtain the desired LMO complexity in~\eqref{eq:lb_Lan}, as evidenced by the success of the two recent methods in~\cite{Theku_20} and~\cite{Asgari_22}, respectively (cf.~Section~\ref{sec:intro}). Furthermore, the counterexample in~\cite[Example 1]{Nest_18} suggests that directly 
using the subgradient of $f$ in the LMO~\eqref{eq:LMO} may result in non-convergence of a particular method to the optimal solution of~\eqref{eq:P} (as the  subgradient is not a descent direction). Since our goal is to achieve the lower bound in~\eqref{eq:lb_holder} uniformly for all $\nu\in[0,1]$, our method also adopts such a ``separation principle''. 
 
Our method is shown in Algorithm~\ref{algo:CG}. As we can see, the LMO of $g$  and the FO of $f$ appear in the $y$-update~\eqref{eq:y-update} and $x$-update~\eqref{eq:x-update}, respectively. In the case of $\nu=0$, our method coincides with the method in~\cite[Algorithm~1]{Asgari_22} when applied to~\eqref{eq:P0}. While certain optimistic primal–dual interpretation was provided in~\cite[Section~2.2]{Asgari_22},  
we provide another interpretation of Algorithm~\ref{algo:CG} through the framework of (linearized) alternating direction method of multipliers (ADMM) (see~\cite{Boyd_11,Ouyang_13,Suzuki_13,Ouyang_15} and references therein). 
Let us define the augmented Lagrangian of~\eqref{eq:Ps} with $f$ replaced by its quadratic approximation at some $\barx\in \bbR^m$: 
\begin{align}
L_{\eta,\rho}(x,y,\lambda;\barx):= f(\barx) + \ip{g_x}{x-\barx} + \frac{\eta}{2}\normt{x-\barx}^2 + g(y) + \ipt{\lambda}{Ay-x} + \frac{\rho}{2}\normt{Ay-x}^2, \label{eq:aug_lag}
\end{align}
where $\eta,\rho>0$ and $g_x\in \partial f(x)$. Consequently, the $x$-update step in~\eqref{eq:x-update} can be written as
\begin{equation}
x_{k+1}  ={\arg\min}_{x\in \R^m}\;\; L_{\eta_k,\rho_k}(x,y_{k+1},\lambda_k;x_k).
\end{equation}
In addition, the $\lambda$-update step~\eqref{eq:dual-update} is 
standard as in the classical augmented Lagragian method~\cite{Rock_76} and ADMM. 
That said, a distinguishing, and quite atypical, feature of Algorithm~\ref{algo:CG} lies in the $y$-update step~\eqref{eq:y-update}, wherein the augmented quadratic term $({\rho}/{2})\normt{Ay-x}^2$ in~\eqref{eq:aug_lag} is missing in the minimization problem. Of course, the motivation here is to utilize the LMO of $g$, but 
this step also makes Algorithm~\ref{algo:CG} essentially fall out of the class of ADMM and many of its variants. As we shall see later, substantially new techniques must be invented to analyze Algorithm~\ref{algo:CG}. Lastly, the output $\bary_K$ is chosen to be a properly weighted average of $\{y_1,\ldots,y_K\}$, and the weights will be given in our convergence rate analysis (cf.~Proposition~\ref{prop:param_choice}). 

Before presenting our analysis of Algorithm~\ref{algo:CG}, we remark that although Algorithm~\ref{algo:CG} shares a similar structure with~\cite[Algorithm~1]{Asgari_22}, 
to achieve the lower complexity in~\eqref{eq:lb_holder} uniformly over $\nu\in[0,1]$, our choices of the algorithmic parameters (i.e., $\{\eta_k\}_{k\ge 0}$ and $\{\rho_k\}_{k\ge 0}$) and the output~$\bary_K$, and more importantly, the analysis techniques, are vastly different from those in~\cite{Asgari_22}. 
More details are provided in Remark~\ref{rmk:AN}. 


\subsection{Convergence Rate Analysis of Algorithm~\ref{algo:CG}}  
\label{sec:analysis}



To begin our analysis, let us note the following fact, which is a restatement of~\cite[Lemma~1]{Nest_15}.

\begin{lemma}\label{lem:inexact}
If $f$ satisfies~\eqref{eq:holder}, then for any $\nu\in[0,1]$, we have  
\begin{equation}
f(y)\le f(x) + \ip{g_x}{y-x} + \frac{L}{2}\norm{y-x}^{2} + \delta, \quad \forall\,g_x\in \partial f(x),\quad \forall\, x,y\in \bbR^m,
\label{eq:inexact_quadratic}
\end{equation}
for some $L>0$ and $\delta:=\delta_\nu(L)\ge 0$. Indeed, if $\nu\in[0,1)$, we can choose any $L>0$ and let 
\begin{align}
\delta_\nu(L):=(1/2)M^{\frac{2}{1-\nu}}L^{-\frac{1+\nu}{1-\nu}}. \label{eq:delta_nu}
\end{align}
If $\nu=1$, we can choose any $L\ge M$ and let $\delta_\nu(L):=0$. 
\end{lemma}

As mentioned in~\cite[Section~2.3.c]{Dev_14}, Lemma~\ref{lem:inexact} suggests that any $f$ satisfying~\eqref{eq:holder} actually admits a first-order $(\delta,L)$-inexact oracle on $\bbR^m$ with exact first-order information (cf.~\cite[Definition~1]{Dev_14}), where $\delta:=\delta_\nu(L)$. 
As we shall see later, this viewpoint indeed provides us some insight into the improved LMO complexity by Algorithm~\ref{algo:CG} for solving~\eqref{eq:P} 
 as compared to the classical FW method. 


Next, we record the following standard result (see e.g.,~\cite[Proof of Definition~3.1]{Kerd_21}). 

\begin{lemma} \label{lem:inexact_lb}
If $f$ satisfies~\eqref{eq:inexact_quadratic}, then 
\begin{equation}
f(y)\ge  f(x) + \ip{g_x}{y-x} +  \frac{1}{2L}\norm{g_y - g_x}^{2} - \delta, \quad \forall\,g_x\in \partial f(x),\; \forall\,g_y\in \partial f(y), \quad \forall\, x,y\in \bbR^m. 
\label{eq:inexact_lb}
\end{equation}
Consequently, we have 
\begin{equation}
\norm{g_y - g_x}^{2} \le L^2\norm{y - x}^{2} + 4L\delta,  \quad \forall\,g_x\in \partial f(x),\; \forall\,g_y\in \partial f(y), \quad \forall\, x,y\in \bbR^m. \label{eq:inexact_grad_norm_lb}
\end{equation}
\end{lemma}

%

Let us now present our analysis of Algorithm~\ref{algo:CG}. 
We shall focus on the general case where $f$ satisfies~\eqref{eq:inexact_quadratic}, and then specialize our results to the case where $f$ is H\"older smooth (cf.~\eqref{eq:holder}).  
Let $(x^*,y^*)$ be an optimal solution of~\eqref{eq:Ps} and $\lambda^*$ an optimal solution of~\eqref{eq:P}. From~\eqref{eq:opt_Fenchel_dual}, we know that $(x^*,y^*,\lambda^*)$ satisfy that 
\begin{equation}
Ay^* = x^*, \quad \lambda^*\in \partial f(x^*), \quad -A^*\lambda^* \in \partial g(y^*). \label{eq:KKT}
\end{equation}
For convenience, define the following Bregman divergences induced by $f$ and $g$, respectively:
\begin{align}
D_f(x,x^*)&:= f(x) - f(x^*) - \ipt{\lambda^*}{x - x^*},    \;\;\;\; \quad \forall\,x\in\bbR^m, \\
 D_g(y,y^*) &:= g(y) - g(y^*) + \ipt{\lambda^*}{Ay - Ay^*}, \quad \forall\,y\in\R^n. 
\end{align}

The starting point of our analysis is the following lemma, which bounds the sub-optimality gap by several key quantities that will be further analyzed. 

\begin{lemma} \label{lem:roadmap}
Let $f$ satisfy~\eqref{eq:inexact_quadratic}. 
 For any $x\in\bbR^m$ and $y\in\bbR^n$, we have 
\begin{align}
 p(y)-p^*  \le 2\Df(x,x^*)+\Dg(y,y^*)+ L\normt{Ay-x}^2 +2\delta.
 \label{eq:objective-decomposition}
\end{align}
\end{lemma}

\begin{proof}
Fix any $x\in\bbR^m$ and $y\in\bbR^n$. Since $Ay^* = x^*$, we have 
\begin{align*}
&f(Ay)+g(y)-\bigl(f(Ay^*)+g(y^*)\bigr)\\
&={f(x)-f(x^*)-\ip{\lambda^*}{x-x^*}}+{g(y)-g(y^*)+\ip{\lambda^*}{Ay-x^*}}+f(Ay)-f(x)-\ip{\lambda^*}{Ay-x}\\
&=
\Df(x,x^*)+\Dg(y,y^*)+\Df(Ay,x) 
+\ip{s-\lambda^*}{Ay-x}, 
\qquad \where\quad s\in\partial f(x).\nt \label{eq:gap_ineq0}
\end{align*}
By Young's inequality, Lemma~\ref{lem:inexact_lb} and that $\lambda^*\in \partial f(x^*)$, we have 
\begin{align}
\ip{s-\lambda^*}{Ay-x} &\le  \frac{1}{2L}\norm{s-\lambda^*}^2+\frac{L}{2}\norm{Ay-x}^2\le D_f(x,x^*) + \delta +\frac{L}{2}\norm{Ay-x}^2. \label{eq:gap_ineq1}
\end{align}
In addition, by~\eqref{eq:inexact_quadratic}, we have 
 \begin{align}
\Df(Ay,x) &\le \frac{L}{2}\norm{Ay-x}^2+\delta.\label{eq:gap_ineq2}
\end{align}
Combining~\eqref{eq:gap_ineq0},~\eqref{eq:gap_ineq1} and~\eqref{eq:gap_ineq2}, we then have~\eqref{eq:objective-decomposition}. 
\end{proof}

Lemma~\ref{lem:roadmap} effectively provides us with a road map to analyze the convergence rate of Algorithm~\ref{algo:CG}: we need to construct some suitable $x\in\bbR^m$ and $y\in\bbR^n$ from  $\{x_{k}\}_{k\ge 0}$ and  $\{y_{k}\}_{k\ge 0}$, respectively, such that 
$\Df(x,x^*)$, $\Dg(y,y^*)$ and $\normt{Ay-x}^2$ all decay at sufficiently fast rates. In fact, our subsequent analysis precisely accomplishes these goals. 

As a key component in our analysis, we identify the following choice of the  Lyapunov sequence: 
\begin{align}
V_k:=\norm{\lambda_{k} -\lambda^*}^2 + \rho_k{\eta_k}\norm{x_k-x^*}^2, \quad \forall\,k\ge 0. 
\end{align}
Indeed, $V_k$ can be interpreted as a linear combination of $\norm{\lambda_{k} -\lambda^*}^2$ and $\norm{x_k-x^*}^2$, which are the squared Euclidean distances from the dual iterate $\lambda_k$ 
and primal iterates $x_{k}$ 
to the corresponding dual and primal optimal solutions 
$\lambda^*$ and $x^*$, respectively. 

The following proposition establishes the recursion of the Lyapunov sequence $\{V_{k}\}_{k\ge 0}$, which lays the foundation of our analysis.

\begin{prop}
Let $f$ satisfy~\eqref{eq:inexact_quadratic}. 
In Algorithm~\ref{algo:CG}, if we choose
\begin{equation}
\eta_k\ge L, \qquad \rho_k\eta_k\ge \rho_{k+1}\eta_{k+1}, \quad \forall\,k\ge 0, \label{eq:param_choices}
\end{equation}
then 
\begin{align}
\rho_k\big(D_f(x_{k+1},x^*)+2D_g(y_{k+1},y^*)\big) \le V_k - \bigg(1+\frac{\rho_k}{2\eta_k}\bigg)V_{k+1}   +{\rho_k^2} D^2  + 3\rho_k\delta, \quad \forall\,k\ge 0. \label{eq:master_recursion}
\end{align}

\end{prop}

\begin{proof}
The first-order optimality condition for~\eqref{eq:x-update} is
\begin{align}
 &\; s_k+\eta_k(x_{k+1}-x_k) -\lambda_{k+1}=s_k-\lambda_k+\eta_k(x_{k+1}-x_k) -\rho_k(Ay_{k+1}-x_{k+1})=0\label{eq:x-optimality0}
\end{align}
Fix any $x\in\bbR^m$. Note that~\eqref{eq:x-optimality0} is equivalent to 
\begin{align}
 &  \ip{s_k+\eta_k(x_{k+1}-x_k) -\lambda_{k+1}}{x-x_{k+1}}=0 \\ 
\Longleftrightarrow \; &  \ip{\lambda_{k+1} -s_k}{x-x_{k+1}}=\eta_k\ip{x_{k+1}-x_k}{x-x_{k+1}}\\
&\qquad \qquad\qquad\qquad\quad\;\, \eqa\frac{\eta_k}{2}\Big(\norm{x-x_k}^2-\norm{x-x_{k+1}}^2-\norm{x_{k+1}-x_k}^2\Big), 
 \label{eq:x-optimality}
 \end{align}
 where in (a) we use the three-point identity:
 \begin{equation}
\ipt{z-y}{x-y} =\frac{1}{2}\Big(\norm{x-y}^2  - \norm{x-z}^2+ \norm{y-z}^2 \Big), \quad \forall\, x,y,z\in\bbR^m. \label{eq:three_pt}
\end{equation}
Write
\begin{align}
\ip{\lambda_{k+1} -s_k}{x-x_{k+1}}= \ip{\lambda_{k+1} }{x-x_{k+1}}+ \ip{s_k}{x_{k+1}-x_k} + \ip{s_k}{x_{k}-x}, \label{eq:lambda_k+1_=0}
\end{align}
and by~\eqref{eq:three_pt} and Lemma~\ref{lem:inexact_lb}, we have  
\begin{align}
\hspace{-.3cm}\ip{\lambda_{k+1} }{x-x_{k+1}}  &= \ip{\lambda_{k} }{x-x_{k+1}}+ \rho_k\ip{Ay_{k+1} - x_{k+1}}{x-x_{k+1}} \\
&= \ip{\lambda_{k}  }{x-x_{k+1}}+ \frac{\rho_k}{2} \big(\norm{x-x_{k+1}}^2  - \norm{x-Ay_{k+1}}^2+ \norm{Ay_{k+1}-x_{k+1}}^2 \big),  \label{eq:lambda_k+1_=}\\
\ip{s_k}{x_{k}-x} &\ge  f(x_k) - f(x) + \frac{1}{2L}\norm{s_k - s}^{2} - \delta, \quad  \forall\,s\in \partial f(x). \label{eq:s_k_lb}
\end{align}
In addition, note that the first-order optimality condition for~\eqref{eq:y-update} is: 
\begin{align}
  g(y_{k+1})-g(y) +\ip{\lambda_k}{Ay_{k+1}-Ay}  \le 0, \quad \forall\, y\in \R^n. 
 \label{eq:y-vi}
\end{align}

Now,  choose 
$(x,y) = (x^*,y^*)$ and $s = \lambda^*$ (cf.~\eqref{eq:KKT}).  Combining~\eqref{eq:x-optimality},~\eqref{eq:lambda_k+1_=0},~\eqref{eq:lambda_k+1_=},~\eqref{eq:s_k_lb} and~\eqref{eq:y-vi},  and noting that $Ay^* = x^*$, 
we have 
\begin{align}
\begin{split}
&\ip{\lambda_{k}  }{Ay_{k+1}-x_{k+1}}+ \frac{\rho_k}{2} \norm{Ay_{k+1}-x_{k+1}}^2  \\
& + f(x_k) - f(x^*)+ \ip{s_k}{x_{k+1}-x_k}  + \frac{1}{2L}\norm{s_k - \lambda^*}^{2} + \frac{\eta_k}{2}\norm{x_{k+1}-x_k}^2 +g(y_{k+1})-g(y^*)\\
&\le \frac{\eta_k}{2}(\norm{x^*-x_k}^2-\norm{x^*-x_{k+1}}^2)  +\frac{\rho_k}{2} (\norm{x^*-Ay_{k+1}}^2-\norm{x^*-x_{k+1}}^2 )  + \delta. 
\end{split}
\label{eq:telescoped0}
\end{align}
In addition, with some simple algebra, we have 
\begin{align}
&\ip{\lambda_{k}  }{Ay_{k+1}-x_{k+1}}+ \frac{\rho_k}{2} \norm{Ay_{k+1}-x_{k+1}}^2 \\
&= \frac{1}{2\rho_k}(  \norm{\lambda_{k} +\rho_k(Ay_{k+1}-x_{k+1})}^2 - \norm{\lambda_{k} }^2)\\
&= \frac{1}{2\rho_k}(  \norm{\lambda_{k+1} }^2 - \norm{\lambda_{k} }^2)\\
&= \frac{1}{2\rho_k}(  \norm{\lambda_{k+1} -\lambda^*}^2 - \norm{\lambda_{k} -\lambda^*}^2+2\ip{\lambda^*}{\lambda_{k+1}-\lambda_{k}})\\
&= \frac{1}{2\rho_k}(  \norm{\lambda_{k+1} -\lambda^*}^2 - \norm{\lambda_{k} -\lambda^*}^2)+\ip{\lambda^*}{Ay_{k+1} - x_{k+1}}\\
&= \frac{1}{2\rho_k}(  \norm{\lambda_{k+1} -\lambda^*}^2 - \norm{\lambda_{k} -\lambda^*}^2)+\ip{\lambda^*}{Ay_{k+1}-Ay^*} - \ip{\lambda^*}{x_{k+1}-x^*}. \label{eq:lambda_k_Ay_k+1}
\end{align}
Substitute~\eqref{eq:lambda_k_Ay_k+1} into~\eqref{eq:telescoped0}, we have 
\begin{align}
\begin{split}
&  f(x_k) - f(x^*)+ \ip{s_k}{x_{k+1}-x_k}  + \frac{1}{2L}\norm{s_k - \lambda^*}^{2} + \frac{\eta_k}{2}\norm{x_{k+1}-x_k}^2\\
&\qquad\qquad\qquad - \ip{\lambda^*}{x_{k+1}-x^*} +g(y_{k+1})-g(y^*)+\ip{\lambda^*}{Ay_{k+1}-Ay^*} \\
\le \;\; &\frac{\eta_k}{2}\norm{x^*-x_k}^2-\frac{\eta_k+\rho_k}{2}\norm{x^*-x_{k+1}}^2  +\frac{\rho_k}{2} \norm{x^*-Ay_{k+1}}^2\\
&\qquad\qquad\qquad\qquad\qquad +\frac{1}{2\rho_k}\Big(  \norm{\lambda_{k} -\lambda^*}^2 - \norm{\lambda_{k+1} -\lambda^*}^2\Big)  + \delta.
\end{split} \label{eq:telescoped1}
\end{align}
Since $\eta_k\ge L$ (cf.~\eqref{eq:param_choices}), we derive two lower bounds on the left-hand side (LHS) of~\eqref{eq:telescoped1}: 
\begin{align*}
\mbox{LHS of~\eqref{eq:telescoped1}}&\ge f(x_{k+1}) - f(x^*) - \ip{\lambda^*}{x_{k+1}-x^*} + \frac{1}{2L}\norm{s_k - \lambda^*}^{2}  +D_g(y_{k+1},y^*) - \delta\\
&\ge D_f(x_{k+1},x^*)+D_g(y_{k+1},y^*) - \delta \nt\label{eq:LHS_lb1}\\
\mbox{LHS of~\eqref{eq:telescoped1}}&\ge f(x_k) - f(x^*)+ \ip{s_k-\lambda^*}{x_{k+1}-x_k}  + \frac{1}{2L}\norm{s_k - \lambda^*}^{2} + \frac{\eta_k}{2}\norm{x_{k+1}-x_k}^2\\
&\hspace{3cm} - \ip{\lambda^*}{x_{k}-x^*}+D_g(y_{k+1},y^*)\\
&\ge f(x_k) - f(x^*)+ \frac{1}{2\eta_k}\norm{s_k-\lambda^*+\eta_k(x_{k+1}-x_k)}^2- \ip{\lambda^*}{x_{k}-x^*}+D_g(y_{k+1},y^*)\\
&= D_f(x_k,x^*)+ \frac{1}{2\eta_k}\norm{\lambda_{k+1}-\lambda^*}^2+D_g(y_{k+1},y^*),\nt\label{eq:LHS_lb2}
\end{align*}
where the last equality follows from~\eqref{eq:x-optimality0}. Averaging~\eqref{eq:LHS_lb1} and~\eqref{eq:LHS_lb2}, we have 
\begin{align}
\mbox{LHS of~\eqref{eq:telescoped1}}&\ge \frac{D_f(x_{k+1},x^*)+D_f(x_k,x^*)}{2}+D_g(y_{k+1},y^*)+ \frac{1}{4\eta_k}\norm{\lambda_{k+1}-\lambda^*}^2 - \frac{\delta}{2}\\
&\ge \frac{1}{2}D_f(x_{k+1},x^*)+D_g(y_{k+1},y^*)+ \frac{1}{4\eta_k}\norm{\lambda_{k+1}-\lambda^*}^2 - \frac{\delta}{2}.  \label{eq:LHS_lb_ave}
\end{align}
Multiplying both sides of~\eqref{eq:telescoped1} by $2\rho_k$, and using~\eqref{eq:LHS_lb_ave} and $\rho_k\eta_k\ge \rho_{k+1}\eta_{k+1}$ in~\eqref{eq:param_choices}, we have 
\begin{align}
&\rho_k(D_f(x_{k+1},x^*)+2D_g(y_{k+1},y^*)) \\
 &\le \rho_k{\eta_k}\norm{x^*-x_k}^2-\rho_k\eta_k\bigg(1+\frac{\rho_k}{\eta_k}\bigg)\norm{x^*-x_{k+1}}^2  +{\rho_k^2} \norm{x^*-Ay_{k+1}}^2\\
 &\qquad +  \norm{\lambda_{k} -\lambda^*}^2 - \bigg(1+\frac{\rho_k}{2\eta_k}\bigg)\norm{\lambda_{k+1} -\lambda^*}^2  + 3\rho_k\delta\\
 &\le \rho_k{\eta_k}\norm{x^*-x_k}^2-\bigg(1+\frac{\rho_k}{2\eta_k}\bigg)\rho_{k+1}\eta_{k+1}\norm{x^*-x_{k+1}}^2  +{\rho_k^2} \norm{x^*-Ay_{k+1}}^2\\
 &\qquad +  \norm{\lambda_{k} -\lambda^*}^2 - \bigg(1+\frac{\rho_k}{2\eta_k}\bigg)\norm{\lambda_{k+1} -\lambda^*}^2  + 3\rho_k\delta\\
 &\le V_k - \bigg(1+\frac{\rho_k}{2\eta_k}\bigg)V_{k+1}   +{\rho_k^2} D^2  + 3\rho_k\delta,
\end{align}
where the last step follows from $x^*= Ay^*\in A(\dom g)$ and~\eqref{eq:def_D}. 
This completes the proof. 
\end{proof}

The next proposition presents suitable choices of the parameters $\{\eta_k\}_{k\ge 0}$ and $\{\rho_k\}_{k\ge 0}$ based on the recursion in~\eqref{eq:master_recursion}.  

\begin{prop}\label{prop:param_choice}
Let $f$ satisfy~\eqref{eq:inexact_quadratic}. 
In Algorithm~\ref{algo:CG}, choose $\eta_k:=\eta\ge L$ for all $k\ge 0$. 
Given $\beta_0>0$ and a positive sequence $\{\alpha_k\}_{k\ge 0}$,  set 
\begin{equation}
\beta_{k+1}:=\beta_k+\alpha_k \quad \andd \quad \rho_k:= 2\eta\frac{\alpha_k}{\beta_k},
\quad \forall\, k\ge 0. 
 \label{eq:beta-definition}
\end{equation}
In particular, by choosing  $\beta_0=1$, $\alpha_k = k+2$ for $k\ge 0$, and $\{\rho_k\}_{k\ge 0}$ according to~\eqref{eq:beta-definition}, then 
\begin{align}
&   D_f(\bar x_k,x^*)+2D_g(\bar y_k,y^*) \le \frac{ V_0 + 16\eta^2D^2k  }{k(k+3)\eta}+ 3\delta, \label{eq:ub_Bregman_div}\\
 &  \norm{\lambda_{k} -\lambda^*}^2 \le  V_{k}\le \frac{ 2(V_0   + 16\eta^2D^2k ) }{(k+1)(k+2)} + 6\eta\delta,
 \end{align}
where
\begin{align}
 \bar x_k:=\frac{\sum_{i=0}^{k-1}\alpha_i x_{i+1}}{\sum_{i=0}^{k-1}\alpha_i} \quad \andd \quad \bar y_k:=\frac{\sum_{i=0}^{k-1}\alpha_i y_{i+1}}{\sum_{i=0}^{k-1}\alpha_i}.
 \label{eq:def_averages}
\end{align}
Consequently, we have 
\begin{align}
\normt{A\bary_k - \barx_k}^2 &\le \frac{7V_0}{\eta^2(k+3)^2} + \frac{48D^2}{k+3} + \frac{7\delta}{\eta}. \label{eq:barr_k}
\end{align}
\end{prop}


\begin{proof}
First of all, note that since $\beta_0=1$ and $\alpha_k = k+2$, we have 
\begin{align}
&\;\beta_k = \frac{(k+1)(k+2)}{2}, \qquad   \sum_{i=0}^{k-1}\alpha_i = \frac{k(k+3)}{2} \quad\andd \quad \frac{\alpha_k}{\beta_k}=\frac{2}{k+1},\qquad\forall\,k\ge 0.  \label{eq:beta_alpha}
\end{align}
Therefore, we have $\rho_k \ge \rho_{k+1}$ for all $k\ge 0$, and the conditions in~\eqref{eq:param_choices} are satisfied. Now, multiplying both sides of~\eqref{eq:master_recursion} by $\beta_k$, then we have 
\begin{align}
2\eta\alpha_k(D_f(x_{k+1},x^*)+2D_g(y_{k+1},y^*)) \le \beta_k V_k - \beta_{k+1}V_{k+1}   + 4\eta^2(\alpha_k^2/\beta_k) D^2  + 6\eta\alpha_k\delta, \quad \forall\,k\ge 0. \nn
\end{align}
Fix any $k\ge 1$. 
Telescoping the above inequality over $i=0,\ldots\,k-1$, then we have 
\begin{align*}
\textstyle 2\eta\sum_{i=0}^{k-1}\alpha_i(D_f(x_{i+1},x^*)+2D_g(y_{i+1},y^*)) \le \beta_0 V_0 - \beta_{k}V_{k}   + 4\eta^2D^2\sum_{i=0}^{k-1}{\alpha^2_i}/{\beta_i}   + 6\eta\delta\sum_{i=0}^{k-1}\alpha_i. 
\end{align*}
Using the convexity of $D_f(\cdot,x^*)$ and $D_g(\cdot,y^*)$ and the definitions of $\bar x_k$ and $\bar y_k$ in~\eqref{eq:def_averages}, we have
\begin{align}
&\quad  0\le 2\eta(D_f(\bar x_k,x^*)+2D_g(\bar y_k,y^*)) \le \frac{\beta_0 V_0- \beta_{k}V_{k}   + 4\eta^2D^2\sum_{i=0}^{k-1}{\alpha^2_i}/{\beta_i}   }{\sum_{i=0}^{k-1}\alpha_i}+ 6\eta\delta\\
\Longrightarrow &\quad  V_{k}\le \frac{\beta_0 V_0   + 4\eta^2D^2\sum_{i=0}^{k-1}{\alpha^2_i}/{\beta_i}   + 6\eta\delta\sum_{i=0}^{k-1}\alpha_i}{\beta_{k}} . 
\end{align} 
Since $\beta_0=1$ and $\alpha_k = k+2$, we have 
\begin{align}
&   D_f(\bar x_k,x^*)+2D_g(\bar y_k,y^*) \le \frac{ V_0 + 16\eta^2D^2k  }{k(k+3)\eta}+ 3\delta,\\
 &  \norm{\lambda_{k} -\lambda^*}^2 \le  V_{k}\le \frac{ 2(V_0   + 16\eta^2D^2k ) }{(k+1)(k+2)} + 6\eta\delta. \label{eq:ub_lambda_k}
\end{align}
%
Next, note that 
\begin{align}
{A\bary_k - \barx_k} =  \frac{\sum_{i=0}^{k-1}\alpha_i (Ay_{i+1} - x_{i+1})}{\sum_{i=0}^{k-1}\alpha_i} &= \frac{\sum_{i=0}^{k-1}(\alpha_i/\rho_i) (\lambda_{i+1} - \lambda_{i})}{\sum_{i=0}^{k-1}\alpha_i}= \frac{\sum_{i=0}^{k-1}\beta_i ((\lambda_{i+1}-\lambda^*) - (\lambda_{i}-\lambda^*))}{2\eta\sum_{i=0}^{k-1}\alpha_i}\nn\\
& = \frac{\beta_{k-1} (\lambda_{k}-\lambda^*)-(\sum_{i=1}^{k-1}\alpha_{i-1}  (\lambda_{i}-\lambda^*) + \beta_0(\lambda_0-\lambda^*))}{2\eta\sum_{i=0}^{k-1}\alpha_i}.\nn
\end{align}
Using the fact that $\normt{\sum_{i=1}^n a_i}^2\le n(\sum_{i=1}^n\normt{ a_i}^2)$, we have 
\begin{align}
\normt{A\bary_k - \barx_k}^2 &\le \frac{2 (\beta_{k-1}^2 \normt{\lambda_{k}-\lambda^*}^2 +\normt{\sum_{i=1}^{k-1}\alpha_{i-1}  (\lambda_{i}-\lambda^*) + \beta_0(\lambda_0-\lambda^*)}^2)}{4\eta^2(\sum_{i=0}^{k-1}\alpha_i)^2}\\
&\le 
\frac{\beta_{k-1}^2 \normt{\lambda_{k}-\lambda^*}^2 +k\sum_{i=1}^{k-1}\alpha_{i-1}^2  \normt{\lambda_{i}-\lambda^*}^2 + k\normt{\lambda_0-\lambda^*}^2}{2\eta^2(\sum_{i=0}^{k-1}\alpha_i)^2}. \label{eq:ub_ave_violation}
\end{align}
By the value of $\{\alpha_k\}_{k\ge 0}$ and $\{\beta_k\}_{k\ge 0}$ in~\eqref{eq:beta_alpha} and~\eqref{eq:ub_lambda_k}, we have 
\begin{align}
\frac{\beta_{k-1}^2 }{(\sum_{i=0}^{k-1}\alpha_i)^2}\normt{\lambda_{k}-\lambda^*}^2 &\le \frac{2V_0 + 32\eta^2D^2k}{(k+3)^2} + 6\eta\delta,\\
k\frac{\sum_{i=1}^{k-1}\alpha_{i-1}^2  \normt{\lambda_{i}-\lambda^*}^2}{(\sum_{i=0}^{k-1}\alpha_i)^2}
&\le \frac{\sum_{i=1}^{k-1}  2V_0 + 32 \eta^2D^2 i + 6\eta\delta(i+1)^2}{k(k+3)^2/4}\\
&\le \frac{  2V_0(k-1) + 16 \eta^2D^2 k(k-1) + \eta\delta k(k+1)(2k+1)}{k(k+3)^2/4}\\
&\le \frac{  8V_0 + 64 \eta^2D^2 (k-1)  }{(k+3)^2}+ 8\eta\delta,\\
k\frac{\normt{\lambda_0-\lambda^*}^2}{(\sum_{i=0}^{k-1}\alpha_i)^2}&\le \frac{4V_0}{k(k+3)^2}\le \frac{4V_0}{(k+3)^2}.
\end{align}
Substituting these bounds into~\eqref{eq:ub_ave_violation}, we then obtain~\eqref{eq:barr_k}.
\end{proof}

\begin{remark}
Note that by the choice of $\{\alpha_k\}_{k\ge 0}$ and $\{\beta_k\}_{k\ge 0}$ in Proposition~\ref{prop:param_choice}, we know that 
\begin{equation}
\rho_k = \frac{4\eta}{k+1}, \quad \forall\, k\ge 0. \label{eq:rho_k} 
\end{equation}
Introducing the auxiliary sequences $\{\alpha_k\}_{k\ge 0}$ and $\{\beta_k\}_{k\ge 0}$ allows us to have a principled way of choosing $\{\rho_k\}_{k\ge 0}$. Indeed, instead of letting $\alpha_k$ grow linearly with $k$, we can choose $\alpha_k = \Theta(k^q)$ for some $q\ge 2$, which leads to different choices of $\{\rho_k\}_{k\ge 0}$, as well as different convergence rates of $\{D_f(\bar x_k,x^*)+2D_g(\bar y_k,y^*)\}_{k\ge 0}$ and $\{V_k\}_{k\ge 0}$. 
In addition, the auxiliary sequences also play important roles in 
defining the weighted averages $\barx_k$ and $\bary_k$ (cf.~\eqref{eq:def_averages}). 
Lastly, it is worth mentioning that the ``auxiliary-sequence'' technique has also been used in analyzing the FW method --- see e.g.,~\cite{Freund_16,Nest_18}. 
\end{remark}

Note that by choosing $x = \barx_k$ and $y = \bary_k$ in Lemma~\ref{lem:roadmap}, and using~\eqref{eq:ub_Bregman_div} and~\eqref{eq:barr_k} in Proposition~\ref{prop:param_choice}, we immediately have the following convergence rate of the sub-optimality gap. 

\begin{theorem}
Let $f$ satisfy~\eqref{eq:inexact_quadratic}. 
In Algorithm~\ref{algo:CG}, under the parameter choices in Proposition~\ref{prop:param_choice}, we have  
\begin{align}
 p(\bary_k)-p^*  \le \frac{9V_0}{k(k+3)\eta}+\frac{ 80\eta D^2  }{k+3}+ 15\delta, \qquad \forall\, k\ge 1. \label{eq:master_rate} 
\end{align}
\end{theorem}

\begin{proof}
By Lemma~\ref{lem:roadmap} and~\eqref{eq:ub_Bregman_div} and~\eqref{eq:barr_k} in Proposition~\ref{prop:param_choice}, we have 
\begin{align}
 p(\bary_k)-p^*  &\le 2(\Df(\barx_k,x^*)+2\Dg(\bary_k,y^*))+ L\normt{A\bary_k-\barx_k}^2 +2\delta\\
 &\le \frac{ 2V_0 + 32\eta^2D^2k  }{k(k+3)\eta}+ 6\delta + \frac{7LV_0}{\eta^2(k+3)^2} + \frac{48LD^2}{k+3} + \frac{7L\delta}{\eta}+2\delta. 
 \end{align}
 Under the choice that $\eta\ge L$, we complete the proof. 
\end{proof}

Note that the convergence rate in~\eqref{eq:master_rate} depends on $V_0$, and holds for any choice of $x^0\in\bbR^m$, $y^0\in\bbR^n$ and $\lambda^0\in\bbR^m$. One may naturally wonder whether the dependence on $V_0$ can be removed. It turns out under the initialization in Algorithm~\ref{algo:CG}, $V_0$ can indeed by upper bounded by some functions of $\eta$, $D$ and $\delta$. 

\begin{lemma} \label{lem:V_0}
Let $f$ satisfy~\eqref{eq:inexact_quadratic}.  
For all $y_0\in\dom g,$ under the initialization in Algorithm~\ref{algo:CG} and the parameter choices in Proposition~\ref{prop:param_choice}, we have 
\begin{align}
V_0 \le 5\eta^2D^2+4\eta\delta. 
\end{align}
\end{lemma}

\begin{proof}
Since $\beta_0=1$ and $\alpha_0=2$, we have  $\rho_0 = 2\eta\alpha_0/\beta_0 = 4\eta$. Since $\lambda_0\in \partial f(x_0)$ and $\lambda^*\in \partial f(x^*)$, by~\eqref{eq:inexact_grad_norm_lb},  we have 
\begin{align}
 V_0 &\le L^2\norm{x_0 - x^*}^{2} + 4L\delta+4\eta^2\norm{x_0-x^*}^2
 \le 5\eta^2 \norm{x_0-x^*}^2 + 4\eta\delta. 
 \label{eq:V0-bound}
\end{align}
Since $\norm{x_0-x^*} = \norm{Ay_0-Ay^*}\le D$, we complete the proof. 
\end{proof}

Based on Lemma~\ref{lem:V_0}, we have the following corollary. 

\begin{corollary}
Let $f$ satisfy~\eqref{eq:inexact_quadratic}. 
In Algorithm~\ref{algo:CG}, under the parameter choices in Proposition~\ref{prop:param_choice}, we have  
\begin{align}
 p(\bary_k)-p^*  
 \le \frac{ 125\eta D^2  }{k+3}+ 24\delta, \qquad \forall\, k\ge 1. \label{eq:rate_eta}
\end{align}
In particular, if $\eta = cL$ for some absolute constant $c\ge 1$, then 
\begin{align}
 p(\bary_k)-p^*   \le \frac{ 125cL D^2  }{k+3}+ 24\delta, \qquad \forall\, k\ge 1. \label{eq:eta=cL}
\end{align}
\end{corollary}

\begin{remark}[Comparison with the FW method]
In the seminal work~\cite{Freund_16}, Freund and Grigas analyzed the convergence rate of the FW method for solving~\eqref{eq:P0}  under the $(\delta,L)$-inexact oracle in~\eqref{eq:inexact_quadratic}. In~\cite[Section 5.2.1]{Freund_16}, they showed that the convergence rate of the FW method 
is of order $O(LD^2/k + k\delta)$. In contrast,~\eqref{eq:eta=cL} shows that the convergence rate of Algorithm~\ref{algo:CG} is of order $O(LD^2/k + \delta)$, which significantly improves the error accumulation term over the FW method. As we shall see shortly, this improved error accumulation precisely explains the improved LMO complexity of Algorithm~\ref{algo:CG} compared to the FW method. 
\end{remark}

Now, let us turn to the case where $f$ is H\"older smooth (cf.~\eqref{eq:holder}). 
 Since $\delta:=\delta_\nu(L)$ depends on $L$, 
 as a standard practice, we can choose $L$ in~\eqref{eq:eta=cL} to achieve the best possible convergence rate in the H\"older-smooth case. This  leads to the following corollary. 

\begin{corollary}\label{cor:holder_rate}
Let $f$ be $(M,\nu)$-H\"older smooth on $\bbR^m$. If $\nu\in[0,1)$,   then for some absolute constant $c>0$, we have 
\begin{align}
p(\bary_K)-p^*  &\le \left({ 125c  D^2  }+ 12M^{\frac{2}{1-\nu}}c^{-\frac{1+\nu}{1-\nu}}\right) K^{-\frac{1+\nu}{2}},\quad \forall\,K\ge 1, \qquad \mbox{if }\;\eta = cK^{\frac{1-\nu}{2}}, \label{eq:holder_rate1}\\
 p(\bary_K)-p^*  &\le \left({ 125c    }+ 12c^{-\frac{1+\nu}{1-\nu}}\right) M D^{\nu+1}K^{-\frac{1+\nu}{2}},\;\;\quad \forall\,K\ge 1, \qquad \mbox{if }\;\eta = cK^{\frac{1-\nu}{2}}MD^{\nu-1}.  \label{eq:holder_rate2}
\end{align}
If $\nu=1$, then by choosing $\eta = cM$ for some absolute constant $c\ge 1$, we have 
\begin{align}
 p(\bary_K)-p^*   \le \frac{ 125cM D^2  }{K+3}, \quad \forall\,K\ge 1.\label{eq:holder_rate_nu=1}
\end{align}
\end{corollary}

\begin{proof}
If $\nu\in[0,1)$, by choosing $\eta = L$ in~\eqref{eq:rate_eta} and the definition of $\delta_\nu(L)$ in~\eqref{eq:delta_nu}, we have 
\begin{align}
 p(\bary_k)-p^*  \le \frac{ 125L D^2  }{k+3}+ 12M^{\frac{2}{1-\nu}}L^{-\frac{1+\nu}{1-\nu}}, \qquad \forall\, k\ge 1. \label{eq:rate_eta}
\end{align}
Since this holds for any $L>0$, we have~\eqref{eq:holder_rate1} and~\eqref{eq:holder_rate2}. If $\nu=1$, then $\delta=0$, and by choosing $\eta=L=cM$ for some $c\ge 1$, we have~\eqref{eq:holder_rate_nu=1}.   
\end{proof}

\begin{remark} \label{rmk:cor_holder}
Let us make some remarks about Corollary~\ref{cor:holder_rate}. 
First, note that from~\eqref{eq:holder_rate2} and~\eqref{eq:holder_rate_nu=1}, 
to find  an $\varepsilon$-optimal solution of~\eqref{eq:P}, 
the LMO complexity of Algorithm~\ref{algo:CG} is 
\begin{equation}
O\Big(M^{\frac{2}{1+\nu}}D^2\varepsilon^{-\frac{2}{1+\nu}}\Big)\qquad \mbox{for all $\nu\in[0,1]$}, \label{eq:optimal_holder_comp}
\end{equation}
which matches the lower bound in~\eqref{eq:lb_holder} (for dimension $n$ sufficiently large). 
 This establishes the optimality of both Algorithm~\ref{algo:CG} and the lower bound in~\eqref{eq:lb_holder}. 
Second, note that for $\nu\in[0,1)$, to obtain the LMO complexity with only the optimal dependence on $\varepsilon$ (but not $M$ and $D$), we could simply choose $\eta = \Theta(K^{{(1-\nu)}/{2}})$ as in~\eqref{eq:holder_rate1}. 
This choice, while resulting in a complexity with sub-optimal dependence on $M$ and $D$,   has the apparent advantage that no knowledge of $M$ and $D$ is required. Third, note that the choices of $\eta$ in Corollary~\ref{cor:holder_rate} require knowing (some of) the problem parameters, including $M$, $D$ and $\nu$, which may be hard to estimate in some situations. 
Providing (partial) remedies for this issue will become the topic of the next section. 
\end{remark}

\begin{remark}\label{rmk:AN}
At the end of Section~\ref{sec:method}, it was mentioned that although being structurally similar, the analysis techniques of Algorithm~\ref{algo:CG}, including choices of algorithmic parameters and output, are vastly different from those in~\cite[Algorithm~1]{Asgari_22}. Let us provide some in-depth discussions. 
At a fundamental level, the analyses in Section~\ref{sec:analysis} and~\cite{Asgari_22} target different function classes of $f$: the former targets $f$ equipped with a first-order $(\delta,L)$-inexact oracle on $\bbR^m$ with exact first-order information (cf.~\eqref{eq:inexact_quadratic}), and the latter targets $f$ being convex Lipschitz on $\bbR^m$ (i.e., $\nu=0$ in~\eqref{eq:holder}). In fact, the analysis in~\cite{Asgari_22} has a similar structure to that of the optimistic mirror descent method in~\cite{Rakhlin_13} (but with novel and important modifications), and critically exploits the global Lipschitz continuity of $f$. (Among others, the global Lipschitz continuity of $f$ is used to i) uniformly bound the dual variables and ii) bound the function-value difference $f(\barx_k) - f(\bary_k)$ via the constraint violation $\normt{\barx_k-\bary_k}$.) In contrast, by judiciously choosing the Lyapunov sequence $\{V_k\}_{k\ge 0}$, we establish the recursion in~\eqref{eq:master_recursion} with a crucial contraction factor in front of $V_{k+1}$, and this forms the cornerstone of our analysis. Consequently, we need to choose $\{\rho_k\}_{k\ge 0}$ as a 
{\em decreasing} sequence (cf.~\eqref{eq:rho_k}), and the output $\bary_K$ as a {\em weighted} average of $\{y_1,\ldots,y_K\}$. By contrast, in~\cite{Asgari_22}, $\{\rho_k\}_{k\ge 0}$ was chosen to be a constant sequence (i.e., $\rho_k\equiv \rho>0$ for all $k\ge 0$) and $\bary_K$ the uniform average of $\{y_1,\ldots,y_K\}$. Indeed, it can be shown that by choosing $\rho_k\equiv \rho$, the convergence rate of Algorithm~\ref{algo:CG} will 
be $\Omega(1/\sqrt{K})$, regardless of the choice of $\eta$ and $\rho.$ 
Finally, we mention that since our analysis also applies to the case where $\nu=0$, it can be regarded as an alternative analysis of~\cite[Algorithm~1]{Asgari_22}. 
\end{remark}


\section{Meta Parameter-Search Algorithms}

\begin{algorithm}[t]
\caption{Fixed-Budget Parameter Search}
\label{algo:meta_fixed}
\begin{algorithmic}[1]
    \State {\bf Input}: $y_{0}\in \dom g$,  
    budget $N$, absolute constant $b\ge 2$ 
    \State {\bf Set}: $K:=\floor{N/(b\log_2 N)}$ and $s:= \floor{(N/K-1)/2}$ \label{item:def_K_S}
    \For{\(j=-s,\ldots,s\)}
        \State Run Algorithm~\ref{algo:CG} for $K$ iterations, with starting point $y_0$ and parameters $\eta_k\equiv\eta^j = 2^j$ and $\rho_k^j$ as given in~\eqref{eq:rho_k}. Denote its output by $\bary_{K}^{j}$.   
\EndFor
   \State {\bf Output}: $\haty_N:= \bary_{K}^{j^*}$, where  $j^*:=\argmin_{-s\le j\le s}\; p(\bary_{K}^{j})$ \label{line:output}
     \end{algorithmic}
\end{algorithm}

As mentioned in Remark~\ref{rmk:cor_holder}, Algorithm~\ref{algo:CG} requires knowing the the problem parameters (such as $M$, $D$ and $\nu$) in order to achieve the optimal LMO complexity in~\eqref{eq:optimal_holder_comp}. Since these parameters may not always be available, 
we propose two meta parameter-search (MPS) algorithms that correspond to fixed and unbounded number of LMO calls, respectively. (In the sequel, we shall refer to the number of LMO calls as ``budget''.) 

The first MPS algorithm is shown in Algorithm~\ref{algo:meta_fixed}. Given a budget $N$, the idea is to run $2s+1$ copies of  Algorithm~\ref{algo:CG} with the same number of iterations $K$  but different values of $\eta$, such that the total number of LMO calls  $(2s+1)K\le N$. 
Specifically, the values of $\eta$ range from $2^{-s}$ to $2^{s}$. 
Since the value of $s$ grows with $N$ (cf.~Line~\ref{item:def_K_S}), as $N$ becomes sufficiently large, one of the copies of Algorithm~\ref{algo:meta_fixed}, say copy $i$, must operate at the ``right'' value of $\eta$ (within a factor of 2). 
Its output, namely $\bary_{K}^{i}$,  in turn provides computational guarantees for the ``best'' output $\haty_N$ (cf.~Line~\ref{line:output}). 
One should note that such a grid-search idea has appeared in the optimization literature multiple times, although in different contexts --- see e.g.,~\cite{Juditsky_14,Roulet_20,Renegar_22}.

We now provide the analysis of Algorithm~\ref{algo:meta_fixed}. For convenience, define 
\begin{equation}
r:=MD^{\nu-1}>0. \label{eq:def_r}
\end{equation}

\begin{theorem} \label{thm:fixed}
Let $f$ be $(M,\nu)$-H\"older smooth on $\bbR^m$.   
In Algorithm~\ref{algo:meta_fixed}, if 
\begin{equation}
N\ge N^*:=\max\left\{4b\log_2(4b),\;(8r^2)^{\frac{1}{b-1+\nu}},\;(8/r^2)^{\frac1b}\right\}\ge  24, \label{eq:N}
\end{equation}
then for all $\nu\in[0,1]$, we have 
\begin{equation}
p(\hat y_N)-p^* \le  262    M D^{\nu+1}\left(\frac{2b\log_2 N}{N}\right)^{\frac{1+\nu}{2}}. \label{eq:haty_N}
\end{equation}
\end{theorem}

\begin{proof}
First, note that since $b>1$, we have $N\ge 4b\log_2(4b)> 8$. Since $\log_2 t\le t$ for all $t>0$,  
we know that
\begin{align}
&\log_2 N  = \log_2 (4b) + \log_2 \big({N}/({4b})\big) \le  {N}/({4b}) + {N}/({4b}) = {N}/({2b}) \label{eq:log_2_N}\\
\Longrightarrow\;\; & {N}/({b\log_2 N})\ge 2 \quad \Longrightarrow\quad K \label{eq:N/blog2N}
\ge 2. 
\end{align}
Define $\eta^* := rK^{\frac{1-\nu}{2}}>0$, so that $\log_2 \eta^*= \log_2 r + \frac{1-\nu}{2} \log_2 K$. We show that under~\eqref{eq:N}, we have
\begin{equation}
-s\le \log_2\eta^*\le s.  \label{eq:cover_s}
\end{equation}
First, since $N>8$ and $b>1$, we have $K\le N/(b\log_2 N)<N/3<N$. 
Indeed, we have 
\begin{align}
&s \ge \frac{N/K-1}{2} -1 \ge \frac{b\log_2 N - 3}{2}  = \frac{1-\nu}{2} \log_2 N + \frac{b-1+\nu}{2} \log_2 N -\frac{3}{2}\\
&\qquad\ge \frac{1-\nu}{2} \log_2 N + \frac{1}{2} (3+2\log_2 r) -\frac{3}{2}\ge \frac{1-\nu}{2} \log_2 K + \log_2 r = \log_2 \eta^*. 
\end{align}
Also, since $N\ge (8/r^2)^{\frac1b}$, we have $b\log_2 N \ge  3-\log_2 r$. Since $K\ge 2$, we have 
\begin{align}
s \ge \frac{b\log_2 N - 3}{2} \ge -\log_2 r = \frac{1-\nu}{2} \log_2 K - \log_2 \eta^*\ge - \log_2 \eta^*. 
\end{align}
Now, by~\eqref{eq:cover_s}, there exists $-s< j \le s$ such that $j-1\le \log_2\eta^*\le j$, which amounts to 
\begin{equation}
\eta^j/2 = 2^{j-1} \le \eta^*\le 2^j =\eta^j\quad \Longleftrightarrow\quad \eta^j = c  \eta^* \quad\mbox{for some } c\in[1,2]. 
\end{equation}
Now, invoking Corollary~\ref{cor:holder_rate}, we have that for all $\nu\in[0,1]$, 
\begin{align}
p(\haty_N)-p^*  &\le  262    M D^{\nu+1}K^{-\frac{1+\nu}{2}}. \label{eq:rate_K} 
\end{align}
From~\eqref{eq:N/blog2N}, we know that ${N}/({2b\log_2 N})\ge 1$, and  hence 
\begin{align}
K\ge {N}/({b\log_2 N}) - 1\ge {N}/({b\log_2 N}) - {N}/({2b\log_2 N}) = {N}/({2b\log_2 N}). \label{eq:lb_K}
\end{align}
Combine~\eqref{eq:rate_K} and~\eqref{eq:lb_K}, we complete the proof. 
\end{proof}

\begin{algorithm}[t]
\caption{Unbounded-Budget Parameter Search}
\label{algo:meta_unlimited}
\begin{algorithmic}
    \State {\bf Input}: $y_{0}\in \dom g$, absolute constant   $b\ge 2$ 
    \State {\bf Set}: $N_0 := \ceil{4b\log_2(4b)}$ 
    \For{\(t=0,1,\ldots\)}
	\State Run Algorithm~\ref{algo:meta_fixed} with starting point $y_0$, budget $N_t:= 2^tN_0$ and $b>1$, 
	and denote its output by $\haty_{t}$. Let $C_t$ the accumulated number of LMO calls after stage $t$. 
\EndFor
 \State {\bf Output}: For any $N\ge 0$, compute $\bart_N:= \max\{t\ge 0: C_t\le N\}$, and output  \vspace{-1ex}
\begin{equation}\vspace{-1ex}
y^{\rm best}_N:= \haty_{t_N^*}, \quad\where \quad t_N^*:={\argmin}_{0\le t\le \bart_N}\;\; p(\haty_{t})
\end{equation} 
\end{algorithmic}
\end{algorithm}

\begin{remark}
Note that the choice of $r$ in~\eqref{eq:def_r} precisely corresponds the choice of $\eta$ in~\eqref{eq:holder_rate2} via the relation $\eta = rK^{{(1-\nu)}/{2}}$. This choice of  $r$ corresponds 
to the ``typical regime of interest'' where both $M$ and $D$ are ``large'', namely $M=\omega(1)$ and $D=\omega(1)$. In this case, from Theorem~\ref{thm:fixed}, we know that 
to find an $\varepsilon$-optimal solution of~\eqref{eq:P}, the LMO complexity of Algorithm~\ref{algo:meta_fixed} is 
\begin{equation}
\begin{split}
&O\Big(\max\Big\{M^{\frac{2}{1+\nu}}D^2\varepsilon^{-\frac{2}{1+\nu}}\log_2\big(MD/\varepsilon\big), \, M^{\frac{2}{b-1+\nu}}D^{-\frac{2(1-\nu)}{b-1+\nu}},\, M^{-\frac{2}{b}}D^{\frac{2(1-\nu)}{b}}\Big\}\Big)\\
= \;&O\Big(M^{\frac{2}{1+\nu}}D^2\varepsilon^{-\frac{2}{1+\nu}}\log_2\big(MD/\varepsilon\big)\Big),  \qquad\qquad \mbox{for all $\nu\in[0,1]$}, 
\end{split}\label{eq:LMO_comp_meta}
\end{equation}
where the equality follows from $b\ge 2$. This reveals a limitation of the MPS approach: we need to pay an additional 
log-factor in the LMO complexity, as compared to the optimal one in~\eqref{eq:optimal_holder_comp}.  
On the other hand, note that if $D=O(1)$, then we can simply omit the dependence of $r$ on $D$ (i.e., let $r:= M$), and the LMO complexity of Algorithm~\ref{algo:meta_fixed} becomes
\begin{equation}
O\Big(M^{\frac{2}{1+\nu}}\varepsilon^{-\frac{2}{1+\nu}}\log_2\big(M/\varepsilon\big)\Big),  \qquad\qquad \mbox{for all $\nu\in[0,1]$}. 
\end{equation}
Of course, the same also applies to the case where $M = O(1)$ or both $M,D=O(1)$. 
\end{remark}

Clearly, with fixed budget $N$, one can only achieve a limited accuracy. 
Sometimes the budget $N$ is not bounded a priori, and in this case arbitrarily high accuracy could in principle be achieved. 
To that end, we present the second MPS algorithm, which is built on Algorithm~\ref{algo:meta_fixed} and shown in Algorithm~\ref{algo:meta_unlimited}. 
The idea behind this algorithm is fairly simple: we run Algorithm~\ref{algo:meta_fixed}  with some initial budget estimate $N_0$, and at each subsequent stage $t$, we run Algorithm~\ref{algo:meta_fixed} with budget $N_t:= 2^tN_0$. For sufficiently large $t$, $N_t$ will exceed the ``minimum requirement'' $N^*$ in~\eqref{eq:N}, and we start to have computational guarantees as given in~\eqref{eq:haty_N}.  The analysis of Algorithm~\ref{algo:meta_unlimited} is shown below. 

\begin{theorem}
\label{thm:unlimited}
Let $N^*$ be given in Theorem~\ref{thm:unlimited}. 
In Algorithm~\ref{algo:meta_unlimited}, if $N\ge 4N^*$,  the for all $\nu\in[0,1]$, we have 
\begin{equation}
    p(y^{\rm best}_N)-p^*    \le 262 M D^{\nu+1}\left(\frac{8b\log_2 N }{N}\right)^{\frac{1+\nu}{2}}.
    \label{eq:rate_unlimited}
\end{equation}
\end{theorem}

\begin{proof}
For convenience, write $t^*:= t_N^*$ and $\bart:=\bart_N$. Let $\tau:= \min\{t\ge 0: N_t\ge N^*\}$. 
If $\tau = 0$, then $N^*=N_0 = N_\tau$; otherwise, we have $N_{\tau}/2 = N_{\tau-1}< N^*\le N_t$, which amounts to $N^*\le N_\tau< 2N^*$. Either way, we have $N_\tau< 2N^*$. 
According to Algorithm~\ref{algo:meta_fixed}, at each stage $t$, we run $2s_t + 1$ copies of Algorithm~\ref{algo:CG}, with each copy run for $K_t$ iterations, where 
$K_t:=\floor{N_t/(b\log_2 N_t)}$ and $s_t:= \floor{(N_t/K_t-1)/2}$.
Thus for all $t\ge 0$, we have 
\begin{align}
C_t = \textstyle\sum_{i=0}^t (2s_t + 1) K_t\le \sum_{i=0}^t N_t = \sum_{i=0}^t 2^t N_0
\le 2^{t+1} N_0 = 2N_t, \label{eq:ub_Ct} 
\end{align}
and hence $C_\tau < 4N_*$. Also, by the definition of $\bart$ and~\eqref{eq:ub_Ct}, we have 
\begin{equation}
C_{\bart}\le N < C_{\bart+1}\le 2N_{\bart+1} = 4N_{\bart}. \label{eq:ub_cbart}
\end{equation}
Now, if $N\ge 4N^*$, then $N>C_\tau$, and we can invoke Theorem~\ref{thm:fixed} to obtain 
\begin{align}
  p(y^{\rm best}_N)-p^*    \le p(\haty_{\bart})-p^* \le  262    M D^{\nu+1}\left(\frac{2b\log_2 N_{\bart}}{N_{\bart}}\right)^{\frac{1+\nu}{2}}< M D^{\nu+1}\left(\frac{8b(\log_2 N - 2)}{N}\right)^{\frac{1+\nu}{2}},\nn
\end{align}
where the last inequality follows from~\eqref{eq:ub_cbart}, $N/4\ge N^*>8$, and the fact that  $x\mapsto \log_2 x/x$ is strictly decreasing on $[e,+\infty)$. 
\end{proof}

From Theorem~\ref{thm:unlimited}, it is clear that the LMO complexity of Algorithm~\ref{algo:meta_unlimited} is the same as that of  Algorithm~\ref{algo:meta_fixed}, which is shown in~\eqref{eq:LMO_comp_meta}. 
Before concluding this section, it is worth mentioning that the 
two MPS algorithms (i.e., Algorithms~\ref{algo:meta_fixed} and~\ref{algo:meta_unlimited}) are not intrinsically tied to Algorithm~\ref{algo:CG}, and can be applied to other algorithms with unknown parameters as well. 


\bibliographystyle{abbrv}
\bibliography{math_opt}

\end{document}